\documentclass[12pt,reqno]{amsart}
\usepackage{amsmath}
\usepackage{amssymb}
\usepackage[left=3cm,top=3cm,right=3cm,bottom=3cm]{geometry}
\usepackage{epsfig}
\usepackage[normalem]{ulem}
\usepackage[usenames,dvipsnames]{color}
\usepackage{todonotes}
\usepackage{bm}

\newcommand{\scr}{\mathcal}

\newcommand{\Ss}{{\mathcal S}}
\newcommand{\Nn}{{\mathcal N}}

\newcommand{\G}{{\mathcal G}}

\newcommand{\Gnp}{\G(n,p)}

\begin{document}
\newtheorem{theorem}{Theorem}[section]
\newtheorem{lemma}[theorem]{Lemma}
\newtheorem{definition}[theorem]{Definition}
\newtheorem{conjecture}[theorem]{Conjecture}
\newtheorem{proposition}[theorem]{Proposition}
\newtheorem{algorithm}[theorem]{Algorithm}
\newtheorem{corollary}[theorem]{Corollary}
\newtheorem{observation}[theorem]{Observation}
\newtheorem{problem}[theorem]{Open Problem}
\newtheorem{remark}[theorem]{Remark}
\newcommand{\noin}{\noindent}
\newcommand{\ind}{\indent}
\newcommand{\om}{\omega}
\newcommand{\I}{\mathcal I}
\newcommand{\pp}{\mathcal P}
\newcommand{\ppp}{\mathfrak P}
\newcommand{\N}{{\mathbb N}}
\newcommand{\LL}{\mathbb{L}}
\newcommand{\R}{{\mathbb R}}
\newcommand{\E}{\mathbb E}
\newcommand{\Prob}{\mathbb{P}}
\newcommand{\eps}{\varepsilon}
\newcommand{\aas}{{a.a.s.}}

\newcommand{\pt}[1]{\textcolor{blue}{#1}}
\newcommand{\pc}[1]{{\bf [~Pawel:\ } {\em \textcolor{blue}{#1}}{\bf~]}}

\newcommand{\SEt}[1]{\textcolor{purple}{#1}}
\newcommand{\SE}[1]{{\bf [~SE:\ } {\em \textcolor{purple}{#1}}{\bf~]}}

\newcommand{\CMc}[1]{{\color{orange}{\bf [~Calum:\ } \color{orange}{\em #1}\color{orange}{\bf~]}}}
\newcommand{\CM}[1]{\textcolor{orange}{#1}}

\title{Localization Game for Random Graphs}

\author{Andrzej Dudek}
\address{Department of Mathematics, Western Michigan University, Kalamazoo, MI, USA}
\thanks{The first author is supported in part by Simons Foundation Grant \#522400.}
\email{\tt andrzej.dudek@wmich.edu}

\author{Sean English}
\address{Department of Mathematics, University of Illinois Urbana-Champaign, Champaign, IL, USA}
\email{\texttt{SEnglish@Illinois.edu}}

\author{Alan Frieze}
\address{Department of Mathematical Sciences, Carnegie Mellon University, Pittsburgh, PA, USA}
\email{alan@random.math.cmu.edu} 
\thanks{The third author is supported in part by NSF grant DMS1362785.}

\author{Calum MacRury}
\address{Department of Computer Science, University of Toronto, Toronto, ON, Canada}
\email{\tt cmacrury@cs.toronto.edu}

\author{Pawe\l{} Pra\l{}at}
\address{Department of Mathematics, Ryerson University, Toronto, ON, Canada}
\email{\texttt{pralat@ryerson.ca}}
\thanks{The fifth author is supported in part by NSERC Discovery Grant.}

\maketitle

\begin{abstract}
We consider the localization game played on graphs in which a cop tries to determine the exact location of an invisible robber by exploiting distance probes. The minimum number of probes necessary per round to locate the robber on a given graph $G$ is the localization number $\zeta(G)$. In this paper, we improve the bounds for dense random graphs determining the asymptotic behaviour of $\zeta(G)$. Moreover, we extend the argument to sparse graphs.
\end{abstract}

\section{Introduction}

Graph searching focuses on the analysis of games and graph processes that model some form of intrusion in a network and efforts to eliminate or contain that intrusion. One of the best known examples of graph searching is the game of Cops and Robbers, wherein a robber is loose on the network and a set of cops attempts to capture the robber. For a book on graph searching see~\cite{book}.

In this paper we consider the \emph{Localization Game} that is related to the well studied Cops and Robbers game. For a fixed integer $k\geq 1$, the localization game with $k$ sensors is a two player combinatorial game played on a graph $G$ which is known to both players. To initialize the game, the \emph{cops} first choose a set $S_1 \subseteq V(G)$
with $|S_1|=k$. The \emph{robber} then chooses a vertex $v\in V(G)$ to start at, whose location on the graph is hidden from the cops. The cops then learn the graph distance between the current position of the robber and the vertices of $S_1$. If this information is sufficient to locate the robber, then the cops win immediately. Otherwise, a new round begins, and the cops now choose another arbitrary subset $S_2 \subseteq V(G)$ of size $k$, based on all the past information available to them. At this point, the robber is allowed to move to any vertex of distance one from its current location, based on $S_1$ \textit{and} $S_2$. The distances of the robber's new location to the vertices of $S_2$ are then presented to the cop, at which point the cops win if these new distance values in conjunction with the previous ones are sufficient to locate the robber. If the cops' information is still insufficient to win the game, then another round begins. These rounds continue iteratively until the cop is able to locate the robber, in which case we say that the cops wins,
or the game proceeds indefinitely, in which case we say that the robber wins. 
We define the localization game in more detail in Section \ref{section definition of the game}.

Given $G$, the localization number, written $\zeta(G)$, is the minimum $k$ so that the cop can eventually locate the robber using sets $W$ of size $k$. The localization game was introduced for one probe ($k=1$) in~\cite{paper19, paper20} and was further studied in~\cite{paper7, paper9, paper5, paper16, Anthony_Bill}.
The localization number is related to the metric dimension of a graph, in a way that is analogous to how the cop number is related to the domination number. The metric dimension of a graph $G$, written $\beta(G)$, is the minimum number of probes needed in the localization game so that the cop can win in one round. It follows that $\zeta(G) \le \beta(G)$, but in many cases this inequality is far from tight. Very recently, \'{O}dor and Thiran \cite{Odor2019} studied a game on $\Gnp$ similar to the localization game, in which the robber is not allowed to move, and the cops place a single sensor in each turn based on the currently revealed distance probes to the robber. In this game, the parameter of interest is the number of rounds necessary to locate the robber, which they refer to as the \textit{sequential metric dimension}. It is clear that the sequential metric dimension is no greater than the metric dimension,
however its relationship with the localization number is uncertain. In Section \ref{sec:results}, our results imply that the localization number can be
strictly less than the sequential metric dimension.

\bigskip

In this paper we present results obtained for the \emph{binomial random graph} $\Gnp$. More precisely, $\Gnp$  is a distribution over the class of graphs with vertex set $[n]$ in which every pair $\{i,j\} \in \binom{[n]}{2}$ appears independently as an edge in $G$ with probability~$p$. Note that $p=p(n)$ may (and usually does) tend to zero as $n$ tends to infinity. We say that $\Gnp$ has some property \emph{asymptotically almost surely} or $\aas$ if the probability that $\Gnp$ has this property tends to $1$ as $n$ goes to infinity.

The localization number for dense random graphs (specifically the diameter two case) was studied in~\cite{Dudek}. The results obtained in~\cite{Dudek} can be summarized as follows (see Section~\ref{sub:notation} for asymptotic notation that we use below). If $d:=p \cdot n=n^{x+o(1)}$ for some $x \in (1/2,1)$, then the following holds a.a.s.\ for $G \in \Gnp$:
$$
(1+o(1)) (2x-1) \frac {n \log n}{d} \le \zeta(G) \le (1+o(1)) f(x) \frac {n \log n}{d},
$$
where 
$$
f(x) := 
\begin{cases}
x & \textrm{if $2/3 < x < 1$} \\
1-x/2 & \textrm{otherwise.} 
\end{cases}
$$
Hence, the order of magnitude of $\zeta(G)$ is determined for this range of $d$. If $d = p \cdot n = n^{1+o(1)}$ and $p \le 1 - 3 \log \log n / \log n$, then the following holds a.a.s.\ for $G \in \Gnp$:
$$
\zeta(G) \sim \frac {2 \log n}{\log (1/\rho)},
$$
where
$$
\rho := p^2 + (1-p)^2.
$$
Thus, the asymptotic behaviour of $\zeta(G)$ was determined in this range.

\bigskip

In this paper, we improve the bounds for dense graphs showing that if $d:=p \cdot n=n^{x+o(1)}$ for some $x \in (1/2,1)$, then a.a.s.\ $\zeta(\Gnp) \sim x n \log n / d$. Our proofs can be easily generalized so we extend our results to cover sparser graphs. The main results are stated in Section~\ref{sec:results}. Notation and some auxiliary observations are presented in Section~\ref{sec:notation}. Section~\ref{sec:perfect} provides a precise definition of the localization game, and a convenient reformulation of the game so that it can be viewed as a perfect information combinatorial game. Finally, lower and upper bounds are proved in Section~\ref{sec:lower} and, respectively, Section~\ref{sec:upper}.

\section{Results}\label{sec:results}
Recall that the asymptotic behaviour of the localization number is already determined for very dense graphs and so we may concentrate on $d = o(n)$.
Our results are slightly stronger than what is stated below but our goal is to summarize the most important consequences. The reader is directed to Sections~\ref{sec:lower} and~\ref{sec:upper} for more details. The first theorem below concentrates on random graphs with diameter $i+1$ and the average degree not too close to the threshold where the diameter drops to $i$. This result follows immediately from Theorem~\ref{theorem lower bound} and Theorem~\ref{theorem upper bound}.

\begin{theorem}\label{thm:main}
Suppose that $d:=p \cdot n$ is such that $\log n \ll d \ll n$. Suppose that $i=i(n) \in \N$ is such that $d^i \ll n$ and $d^{i+1}/n-2\log n \to \infty$. 
Then, the following holds a.a.s.\ for $G \in \Gnp$: 
\[
\left(\log d - 3\log\log n\right)\frac{n}{d^i} \le \zeta(G)\le (1+o(1)) \left(\log d + 2 \log\log n\right)\frac{n}{d^i}\,.
\]
As a result, if $d \ge (\log n)^{\omega}$ for some $\omega = \omega(n) \to \infty$ as $n \to \infty$, then
$$
\zeta(G) \sim \frac{n \log d}{d^i}\,.
$$
In particular, if there exists $i\in \N$ such that $d=n^{x+o(1)}$ for some $x\in (\frac{1}{i+1},\frac{1}i)$, then 
$$
\zeta(G) \sim \frac{x n \log n}{d^i}\,.
$$
\end{theorem}

Before we move to our next result, let us mention the relationship between $\zeta(G)$ and $\beta(G)$. The bounds for $\beta(G)$ obtained in~\cite{metric_dimension} are quite technical but for the range of $d$ covered by Theorem~\ref{thm:main} we see that the following holds a.a.s.\ for $G \in \Gnp$: 
$$
(1+o(1)) \frac{n \log (d^i) }{d^i} \le \beta(G) \le (1+o(1)) \frac{n \log n}{d^i}\,.
$$

In~\cite{Dudek}, it was conjectured that when $d=n^{x+o(1)}$ for some $x\in (2/3,1)$, we have that $\zeta(G) < \beta(G)$. In this case, since $i=1$, our asymptotic bound on $\zeta(G)$ matches with the lower bound on $\beta(G)$ given above, and so in order to prove or disprove such a conjecture, one would need to obtain new bounds on the metric dimension. 

On the other hand, for sparser graphs (of diameter at least $3$; $i \ge 2$), it follows that $\zeta(G) < \beta(G)$. In fact, if $d=n^{x+o(1)}$ for some $x\in (\frac{1}{i+1},\frac{1}i)$, $i \in \N \setminus \{1\}$, then a.a.s.\ $i + o(1) \le \beta(G) / \zeta(G) \le 1/x + o(1) < i+1$ and so these two graph parameters are a multiplicative constant far away from each other (the ratio is roughly equal to the diameter of the graph). Moreover, for very sparse graphs, say for example $d = \log^6 n$, a.a.s.\ $\zeta(G) = \Theta( n \log \log n / d^i)$ whereas $\beta(G) = \Theta( n \log n / d^i)$, implying that for such value of $d$, $\zeta(G)=o(\beta(G))$. We remark
that the results of \'{O}dor and Thiran \cite{Odor2019} show that the sequential metric dimension of $G \in \Gnp$ is $\aas$ within a constant factor of $\beta(G)$.
As such, $\zeta(G)$ is also dominated by the sequential metric dimension of $G$ for this parameter range of $\Gnp$.

\bigskip

We are less precise once we get closer to the threshold where the diameter drops from $i+1$ to $i$. If $c = c(n) := d^{i}/n = \Theta(1)$, then we only determine the order of $\zeta(G)$. When $c \to \infty$ as $n \to \infty$, then the upper bound for $\zeta(G)$ does not match the corresponding lower bound. Thus, determining the behaviour of the localization number when $c \to \infty$ remains an open problem. Below, we state the result for $c = \Theta(1)$ and we direct the reader for more details on the case when $c \to \infty$ to Sections~\ref{sec:lower} and~\ref{sec:upper}. This result follows immediately from Theorem~\ref{theorem lower bound} and Theorem~\ref{theorem upper bound2}.

\begin{theorem}\label{thm:main2}
Suppose that $d:=p \cdot n$ is such that $\log^3 n \ll d \ll n$. Suppose that $i=i(n) \in \N$ is such that $c = c(n) := d^i / n \to A \in \R_+$. 
Then, the following holds a.a.s.\ for $G \in \Gnp$: 
\[
\left(\log d - 3\log\log n\right)\frac{1}{A} \le \zeta(G)\le (1+o(1)) \left(\log d + 2 \log\log n\right)\frac{e^A}{1-e^{-A}}\,.
\]
As a result, if $d \ge (\log n)^{3+\epsilon}$ for some $\epsilon > 0$, then 
$$
\zeta(G) = \Theta \left( \frac{n \log d}{d^i} \right)\,.
$$
\end{theorem}

\section{Notation and Probabilistic Preliminaries}\label{sec:notation}

In this section we give a few preliminary results that will be useful for the proof of our main result. First, we introduce standard asymptotic notation, then we state a specific instance of Chernoff's bound that we will find useful. Finally, we mention some specific expansion properties that $\Gnp$ has and state the well-known result about the diameter of $\Gnp$.

\subsection{Notation and Convention}\label{sub:notation}
 
Given two functions $f=f(n)$ and $g=g(n)$, we will write $f=O(g)$ if there exists an absolute constant $\alpha$ such that $f
\leq \alpha \cdot g$ for all $n$, $f=\Omega(g)$ if $g=O(f)$, $f=\Theta(g)$ if $f=O(g)$ and $f=\Omega(g)$, and we write $f=o(g)$ or $f\ll g$ if the limit $\lim_{n\to\infty} f/g=0$. In addition, we write $f=\omega(g)$ or $f\gg g$ if $g=o(f)$, and unless otherwise specified, $\omega$ will denote an arbitrary function that is $\omega(1)$, assumed to grow slowly. We also will write $f\sim g$ if $f=(1+o(1))g$. 

For a vertex $v \in V$ of some graph $G=(V,E)$, let $\Ss(v,i)$ and $\Nn(v,i)$ denote the set of vertices at distance exactly $i$ from $v$ and the set of vertices at distance at most $i$ from $v$, respectively. For any $V' \subseteq V$, let $\Ss(V',i) = \bigcup_{v \in V'} \Ss(v,i)$ and $\Nn(V',i) = \bigcup_{v \in V'} \Nn(v,i)$.

Through the paper, all logarithms with no subscript denoting the base will be taken to be natural. Finally, as typical in the field of random graphs, for expressions that clearly have to be an integer, we round up or down but do not specify which: the choice of which does not affect the argument.

\subsection{Concentration inequalities}

Throughout the paper, we will be using the following concentration inequality. Let $X \in \textrm{Bin}(n,p)$ be a random variable with the binomial distribution with parameters $n$ and $p$. Then, a consequence of Chernoff's bound (see e.g.~\cite[Corollary~2.3]{JLR}) is that 
\begin{equation}\label{eq:chern}
\Prob( |X-\E X| \ge \eps \cdot \E X) ) \le 2\exp \left( - \frac {\eps^2  \cdot \E X}{3} \right)  
\end{equation}
for  $0 < \eps < 3/2$. 

\subsection{Expansion properties}

In this paper, we focus on dense random graphs, that is, graphs with average degree asymptotic to $d := p \cdot n \gg \log n$. Such dense random graphs will have some useful expansion properties that hold $\aas$.
We will use the following two technical lemmas. The first one is proven in~\cite{metric_dimension} but we include the proof for completeness. 

\begin{lemma}[\cite{metric_dimension}]\label{lem:gnp exp}
Let $\omega=\omega(n)$ be a function tending to infinity with $n$ such that $\omega \le (\log n)^4 (\log \log n)^2$.  Then the following properties hold a.a.s.\ for $G=(V,E) \in \G(n,p)$.
Suppose that $\omega \cdot \log n \le d:=p \cdot n = o(n)$. Let $V' \subseteq V$ with $|V'| \le 2$ and let $i=i(n) \in \N$ be such that $d^i = o(n)$. Then,
$$
\left| \Ss(V',i) \right| = \left(1+O \left( \frac {1}{\sqrt{\omega}}\right) + O\left( \frac{d^i}{n} \right) \right) d^i |V'|.
$$
In particular, for every $x,y \in V$ ($x \neq y$) we have that
$$
\left| \Ss(x,i) \setminus \Ss(y,i) \right| = \left(1+O \left( \frac {1}{\sqrt{\omega}}\right) + O\left( \frac{d^i}{n} \right) \right) d^i.
$$
\end{lemma}

For the next lemma we need to assume that our random graph is slightly denser, namely, that $d := p \cdot n \gg \log^3 n$.

\begin{lemma}\label{lem:gnp exp2}
Let $\omega'=\omega'(n)$ be a function tending to infinity with $n$ such that $\omega' \le (\log n)^2 (\log \log n)^2$.  Then the following properties hold a.a.s.\ for $G=(V,E) \in \G(n,p)$.
Suppose that $\omega' \log^3 n \le d:=p \cdot n = o(n)$. Suppose that $i=i(n) \in \N$ is such that $c = c(n) := d^i/n = \Omega(1)$ and $c \le 3 \log n$. Then, for every $x,y \in V$ ($x \neq y$) we have that
$$
\left| \Ss(x,i) \setminus \Ss(y,i) \right| = \left( 1 + O \left( \frac {1}{\sqrt{\omega'}} \right) \right) n (1-e^{-c}) e^{-c},
$$
provided $c \le \log n - 4 \log \log n$. For $\log n - 4 \log \log n \le c \le 3 \log n$, we have that
$$
\left| \Ss(x,i) \setminus \Ss(y,i) \right| = O \left( \log^4 n \right).
$$
\end{lemma}

\begin{proof}[Proof of Lemma~\ref{lem:gnp exp}]
We will show that a.a.s.\ for every $V' \subseteq V$ with $|V'| \le 2$ and $i \in \N$, we have the desired concentration for $|\Ss(V',i)|$, provided that $d^i = o(n)$. The statement for each pair of distinct vertices $x,y$ then follows immediately via a deterministic argument.

In order to investigate the expansion property of neighbourhoods, let $Z \subseteq V$, $z=|Z|=o(n/d)$, and consider the random variable $X = X(Z) = |\Nn(Z,1)|$. We will bound $X$ stochastically. There are two things that need to be estimated: the expected value of $X$, and the concentration of $X$ around its expectation.

Since for $x=o(1)$ we have that $(1-x)^z = e^{-x \cdot z \cdot (1+O(x))}$ and $e^{-x}=1-x+O(x^2)$, it is clear that
\begin{eqnarray}
\E [X] &=& n - \left(1- \frac {d}{n} \right)^z (n-z) \nonumber \\
&=& n - \exp \left( - \frac {dz}{n} (1+O(d/n)) \right) (n-z) \nonumber \\
&=& d \cdot z \cdot (1+O(dz/n)), \label{eq:expX}
\end{eqnarray}
provided $d \cdot z = o(n)$. It follows from Chernoff's bound~(\ref{eq:chern}), applied with $\eps = 2/{\sqrt{\omega}}$, that the expected number of sets $V'$ with $|V'|\leq 2$ satisfying  
\[
\big| |\Nn(V',1)| - \E\left[\,|\Nn(V',1)|\,\right] \big| > \eps d|V'|
\] 
is at most 
\[
\sum_{z \in \{1,2\}} 2 n^z \exp \left( - \frac {\eps^2 z d }{3+o(1)} \right) \le \sum_{z \in \{1,2\}} 2 n^z \exp \left( - \frac {\eps^2 z \omega \log n}{3+o(1)} \right) = o(1),
\]
since $d \geq \omega \log n$. Hence the statement holds $\aas$ for $i=1$. 

We now estimate the cardinalities of $\Nn(V',i)$ up to the $i^{th}$ iterated neighbourhood, provided $d^i = o(n)$ and thus $i = O(\log n /\log \log n)$. It follows from~(\ref{eq:expX}) and~(\ref{eq:chern}) (with $\eps = 4 (\omega \cdot |Z|)^{-1/2}$) that 
with probability at least $1-n^{-3}$ it holds that
$$
|\Nn(Z,1)| = d \cdot |Z| \left( 1+ O \left( d|Z|/n \right) + O\left((\omega |Z|)^{-1/2} \right) \right),
$$
assuming $\omega \log n / 2 \le |Z| = o(n/d)$. We emphasize that the bounds in $O()$ are uniform. As we want a result that holds a.a.s., we may assume this statement holds deterministically, since there are only $O(n^2 \log n)$ choices for $V'$ and $i$. Given this assumption, we have good bounds on the ratios of the cardinalities of $\Nn(V',1)$, $\Nn(\Nn(V',1),1) = \Nn(V',2)$, and so on. Since $i=O(\log n / \log \log n)$ and $\sqrt{\omega} \le (\log n)^2 (\log \log n)$, the cumulative multiplicative error term is
\begin{align*}
(1+&O(d/n) + O(1/\sqrt{\omega})) \prod_{j=2}^i \left( 1+ O \left( d^j/n \right) + O\left( \omega^{-1/2} d^{-(j-1)/2} \right) \right) \\
&= (1+O(1/\sqrt{\omega}) + O(d^i/n) )  \prod_{j=7}^{i-3} \left( 1+ O \left( \log^{-3} n \right) \right) = (1+O(1/\sqrt{\omega}) + O(d^i/n) ),
\end{align*}
and the proof is complete.
\end{proof}

\begin{proof}[Proof of Lemma~\ref{lem:gnp exp2}]
Fix any $x,y\in V$ ($x \neq y$). Since $d=o(n)$ and $d^i = \Omega(n)$, it follows that $i \ge 2$. We expose edges around vertices $x$ and $y$ to get $\Nn(\{x,y\},i-1)$. Note that $d^{i-1} = d^i/d = cn/d = O(n \log n / d) = O(n / (\omega' \log^2 n)) = o(n)$. Hence, by Lemma~\ref{lem:gnp exp} applied with $\omega = \omega' \log^2 n$, we may assume that
\begin{eqnarray*}
\left| \Ss(x,i-1) \setminus \Ss(y,i-1) \right| &=& \left(1+O \left( \frac {1}{\sqrt{\omega}}\right) + O\left( \frac{d^{i-1}}{n} \right) \right) d^{i-1} \\
&=& \left(1+O \left( \frac {1}{\sqrt{\omega'} \log n}\right) + O\left( \frac{1}{\omega' \log n} \right) \right) d^{i-1} \\
&=& \left(1+O \left( \frac {1}{\sqrt{\omega'} \log n}\right) \right) d^{i-1}.
\end{eqnarray*}
Similarly,
$$
\left| \Ss(y,i-1) \right| = \left(1+O \left( \frac {1}{\sqrt{\omega'} \log n}\right) \right) d^{i-1}.
$$
Let $X = X(x,y) = \left| \Ss(x,i) \setminus \Ss(y,i) \right|$. It is clear that $v \in V \setminus \Nn(\{x,y\},i-1)$ belongs to $\Ss(x,i) \setminus \Ss(y,i)$ if and only if $v$ has a neighbour in $\Ss(x,i-1) \setminus \Ss(y,i-1)$ but has no neighbour in $\Ss(y,i-1)$. It follows that
\begin{eqnarray*}
\E [X] &=& \Big( n - |\Nn(\{x,y\},i-1)| \Big) \left( 1 - (1-p)^{\left| \Ss(x,i-1) \setminus \Ss(y,i-1) \right|} \right) (1-p)^{\left| \Ss(y,i-1) \right|}.
\end{eqnarray*}
Since 
\begin{eqnarray*}
(1-p)^{\left(1+O \left( \frac {1}{\sqrt{\omega'} \log n}\right) \right) d^{i-1}} &=& \exp \left( - \left(1+O \left( \frac {1}{\sqrt{\omega'} \log n}\right) \right) \frac {d^{i}}{n} \right) \\
&=& \exp \left( - c + O \left( \frac {c}{\sqrt{\omega'} \log n}\right) \right) \\
&=& e^{-c} \exp \left( O \left( \frac {1}{\sqrt{\omega'}} \right) \right) \\
&=& e^{-c} \left( 1 + O \left( \frac {1}{\sqrt{\omega'}} \right) \right),
\end{eqnarray*}
we get that
\begin{eqnarray*}
\E [X] &=& \left( 1 + O \left( \frac {d^{i-1}}{n} \right) \right) n (1-e^{-c}) e^{-c} \left( 1 + O \left( \frac {1}{\sqrt{\omega'}} \right) \right) \\
&=& \left( 1 + O \left( \frac {1}{\sqrt{\omega'}} \right) \right) n (1-e^{-c}) e^{-c}.
\end{eqnarray*}

Suppose first that $c \le \log n - 4 \log \log n$ so that $\E [X] \ge (1+o(1)) \log^4 n.$ 
It follows from Chernoff's bound~(\ref{eq:chern}), applied with $\eps = 1/{\sqrt{\omega'}} \ge (\log n)^{-1} (\log \log n)^{-1}$, that 
$$
X = \left( 1 + O \left( \frac {1}{\sqrt{\omega'}} \right) \right) n (1-e^{-c}) e^{-c}
$$ 
with probability at least 
$$ 
1 - \exp\Big( - \Theta(\eps^2 \E[X]) \Big) = 1 - \exp \Big( - \Omega( (\log n)^2 / (\log \log n)^2) \Big) = 1-o(n^{-2}).
$$
The desired property holds by the union bound taken over all pairs $x,y$.

For $\log n - 4 \log \log n < c \le 3 \log n$, we have $\E[X] \le (1+o(1)) \log^4 n$. We may couple the binomial random variable $X$ with another random variable $Y \ge X$ with expectation equal to $(1+o(1)) \log^4 n$. Then, we may use Chernoff's bound~(\ref{eq:chern}) with, say, $\eps = 1$ to get that with the desired probability $X \le Y \le (2+o(1)) \log^4 n$. (Alternatively, one could use a more general version of Chernoff's bound that allows $\eps \ge 3/2$.) The desired bound for $X=X(x,y)$ holds \aas\ for all pairs of $x,y$.
\end{proof}

\begin{remark}\label{remark_cond}
Let us indicate a high level overview of how we are going to apply Lemma~\ref{lem:gnp exp} (or other properties of $\G(n,p)$ that hold $\aas$) throughout the paper. This is a standard technique in the theory of random graphs but it is quite delicate. We wish to use the expansion properties guaranteed $\aas$ by Lemma~\ref{lem:gnp exp}, but we also wish to avoid working in a conditional probability space, as doing so would make the necessary probabilistic computations intractable.
Thus, we will work in the unconditional probability space of $\G(n,p)$, but we will provide an argument which assumes $\G(n,p)$ 
has the expansion properties of Lemma~\ref{lem:gnp exp}. Since these properties hold $\aas$, the probability of the set of outcomes in which our argument does not apply to is $o(1)$, and thus can be safely excised at the end of the argument. We provide more details as they become relevant in our specific applications
of this technique, namely, in Lemma~\ref{lemma diametrically opposed}.


\end{remark}

\subsection{Diameter of $\Gnp$}

We will use the following well-known result.

\begin{lemma}[\cite{bol}, Corollary 10.12]\label{lem:diameter}
Suppose that $d := p \cdot n \gg \log n$ and 
\[
d^{i+1}/n - 2 \log n \to \infty \text{ \ \ \ \ and \ \ \ \ } d^{i}/n - 2 \log n \to -\infty.
\]
Then the diameter of $G \in \G(n,p)$ is equal to $i+1$ a.a.s.
\end{lemma}

\section{Definition and Reformulation of the Localization Game}\label{sec:perfect}

In this section, we will first provide a precise definition of the localization game, and then provide a reformulation of the game that will be easier to deal with when proving our results.

\subsection{Definition of the Localization Game}\label{section definition of the game}

Let $G=(V,E)$ be a connected graph. Given a set $S\subseteq V$ of size $k$, $S=\{s_1,s_2,\dots,s_k\}$, and a vertex $v\in V$, we say the $S$-\emph{signature} of $v$ is the vector $\mathbf{d}=\mathbf{d}(S,v)=(d_1,d_2,\dots,d_k)$ where $d_i=d(s_i,v)$ is the distance from $s_i$ to $v$ for each $1\leq i\leq k$. Then the \emph{localization game with $k$ sensors} is a game played with two players, the \emph{cops} and the \emph{invisible robber}. In the first round, the cops choose a set $S_1\subseteq V$, $|S_1|=k$ (called the \emph{sensor locations}), the robber chooses any vertex $v_1\in V$, and then the cops receive the $S_1$-signature of $v_1$, say $\mathbf{d}_1$. If the $S_1$-signature of $v_1$ is sufficient to determine the location of the robber, the cops win, otherwise the game continues to the next round. Then, in round $i$, the cops choose a new set $S_i\subseteq V$, and the robber chooses a vertex $v_i\in \mathcal{N}(v_{i-1},1)$ as her new location, and the cops learn the $S_i$-signature of $v_i$, say $\mathbf{d}_i$. 

While playing the localization game, both the cops and the robber are aware of the underlying graph and all the previous cops' moves. However, the cops are not aware of the exact location of the robber, but the robber is aware of every move they have made. Thus, the robber has perfect information in the localization game, while the cops do not.

We call the sequence $(\mathbf{d}_1,\mathbf{d}_2,\dots,\mathbf{d}_i)$ the \emph{info trail of the walk $(v_1,v_2,\dots,v_i)$ with respect to sensor locations $(S_1,S_2,\dots,S_i)$}. Then the cops win in round $i$ if the info trail of the robber is sufficient to determine the location of the robber, and otherwise the game proceeds to round $i+1$. More precisely, the cops win in round $i$ if for every two walks (we will assume here that $G$ is reflexive so consecutive vertices in a walk can be equal), $W=(w_1,w_2,\dots,w_i)$ and $X=(x_1,x_2,\dots,x_i)$, both with info trail $(\mathbf{d}_1,\mathbf{d}_2,\dots,\mathbf{d}_i)$ with respect to $(S_1,S_2,\dots,S_i)$, we have $w_i=x_i$. 

The \emph{localization number} of the graph $G$, denoted $\zeta(G)$ is defined to be the least integer $k$ such that the cops can win the localization game with $k$ sensors in finite time, regardless of the strategy of the robber. It is worth noting that since the definition of the localization number requires the cops to be able to win regardless of the strategy of the robber, the fact that the robber has perfect information is somewhat arbitrary; regardless of the information given to the robber, $\zeta(G)$ does not change. 

\subsection{Reformulation of the Game with Perfect Information for the Cops}

Since the definition of localization number requires the cops to be able to win in finite time regardless of the strategy of the robber, we can view this problem equivalently as follows: when the cops choose $S_1$, we partition the vertex set $V$ into $R^1_{1}\cup R^1_{2}\cup \ldots \cup R^1_{\ell_1}$ such that the sets $R^1_{j}$ are the equivalence classes of vertices in $V$ that have the same $S_1$-signature for $1\leq j\leq \ell_1$. Then, instead of choosing a specific location, the robber can choose some equivalence class $R^1_{j_1}$. Then once the cops choose $S_2$, we partition the set $\Nn(R^1_{j_1},1)$ into equivalence classes $R^2_{1}\cup R^2_{2}\cup \ldots \cup R^2_{\ell_2}$ so that every vertex in $R^2_{j}$ has the same $S_2$-signature. Then the robber chooses a set $R^2_{j_2}$. Iteratively, in round $i$, once the cops choose $S_i$, this gives the partition $\Nn(R^{i-1}_{j_{i-1}},1)=R^i_{1}\cup R^i_{2}\cup \ldots \cup R^i_{\ell_i}$ with every vertex in $R^i_{j}$ having the same $S_i$ signature, then the robber chooses some $R^i_{j_i}$. In this version of the game, the cops win in round $i$ if the robber is forced to choose a set $R^i_{j_i}$ with only one vertex, that is, $|R^i_{j_i}|=1$. In this reformulation, both players have perfect information. In particular, it is a combinatorial game and so one of the players must have a winning strategy; that is, a
strategy which wins against all of the other player's strategies simultaneously.

It can be seen that these two formulations of the localization game are equivalent in the sense that if the robber performs the walk $(v_1,v_2,\dots,v_i)$ in response to sensor locations $(S_1,S_2,\dots,S_i)$, this is equivalent to the robber choosing sets $(R^1_{j_1},R^2_{j_2},\ldots,R^i_{j_i})$, and if there is enough information to determine that the robber is at $v_i$ at time $i$, it must be because $R^i_{j_i}=\{v_i\}$ has only one element. Conversely, if the robber chooses sets $(R^1_{j_1},R^2_{j_2},\dots,R^i_{j_i})$ in response to the cop choosing sensor locations $(S_1,S_2,\dots,S_i)$, then there exists at least one walk $(v_1,v_2,\dots,v_i)$ with $v_k\in R^k_{j_k}$ for each $1\leq k\leq i$, and if $|R^i_{j_i}|=1$, we must have $R^i_{j_i}=\{v_i\}$ and every walk that shares an info trail with $(v_1,v_2,\dots,v_i)$ must have terminal vertex $v_i$, so the cops locate the robber. Thus, the two formulations are equivalent in terms of the value of $\zeta(G)$, and throughout the rest of the paper we will work with this perfect information reformulation of the localization game.

Throughout the rest of the paper, when a robber chooses a set $R^i_{j_i}$ on turn $i$, we will denote that set simply by $R^i$ for the remainder of the game.

\section{Lower Bound}\label{sec:lower}

In this section, we prove our lower bound. Our upper bound of $o(n \log \log n)$ for $d^i$ is not best possible. For $d^i = \Omega(n)$ we do not manage to determine the asymptotic behaviour of $\zeta(G)$ and so we content ourselves with a slightly weaker bound. 

\begin{theorem}\label{theorem lower bound}
Suppose that $d:=p \cdot n$ is such that $\log n \ll d \ll n$. Let $i = i(n) \in \N$ be the largest integer $i$ such that $d^i \ll n \log \log n$.
Suppose that $d^{i+1}/n - 2 \log n \to \infty$.
Then the following holds a.a.s.\ for $G \in \Gnp$: 
\[
\zeta(G)\geq \left(\log d - 3\log\log n\right)\frac{n}{d^i}\,.
\]
\end{theorem}

We say that two vertices $v,w$ are {\em diametrically opposite}, provided the the distance between them equals the diameter of the graph. 
Our goal is to bound the number of vertices which are diametrically opposite to all the vertices in the set of sensors. 
In this way, we prove Theorem \ref{theorem lower bound} by arguing that a greedy strategy works for the robber, provided
the number of sensors of the cop is sufficiently small. Specifically, the strategy of the robber is to maintain maximum distance from all of the sensors.
In doing so, they are able to evade the cop indefinitely, and thus win the game.

In order to show that this greedy strategy works, first recall that by Lemma~\ref{lem:diameter}, $\aas$ $G$ has diameter $i+1$.
The lemma below strengthens this result, by ensuring that no matter where the sensors of the cop are placed,
there always exists a large selection of vertices which are simultaneously at distance $i+1$ from all of the sensors.

\begin{lemma}\label{lemma diametrically opposed}
Suppose that $d:=p \cdot n$ is such that $\log n \ll d \ll n$. Let $i = i(n) \in \N$ be the largest integer $i$ such that $d^i \ll n \log \log n$.
Let 
$$
s:=\Big( \log d -3\log\log n \Big)\frac{n}{d^i} \quad \text{ and } \quad r := \frac {n \log^3n}{d}.
$$
Then the following holds a.a.s.\ for $G = (V,E) \in \Gnp$: 
for every set $S\subseteq V$ with $|S|=s$, we have that
\[
|V \setminus \Nn(S,i)| = n-|\Nn(S,i)|\geq r/2.
\]
\end{lemma}

\begin{proof}
Let $S\subseteq V$ be a set of size $s$. We will sequentially expose edges incident to $S$ in order to determine $\Nn(S,i-1)$.
Specifically, for $v \in S$, we expose edges via the following procedure:
\begin{itemize} 
    \item For $j=0, \ldots , i-2$ do:
    \item[] \quad expose the edges of $G$ incident to $\Nn(v,j)$ which are still unexposed.
\end{itemize}

Let us denote $\scr{E}_{v}$ as the information regarding the edges 
of $G$ we reveal by following this procedure for a fixed $v \in S$; that is,
$\scr{E}_v$ corresponds to the vertex pairs of $\binom{V}{2}$ which
are exposed, as well as indications as to whether each exposed
vertex pair is an edge of $G$. Similarly,
denote $\scr{E}$ as the information revealed
after following this procedure for each $v \in S$.

We first note that $\scr{E}_v$ is sufficient to determine $\Ss(v,j)$
for each $v \in S$ and $j=1, \ldots i-1$. At this point, let us say that the set $S$ is \textit{good}, provided for each
$v \in S$ and $j=1, \ldots ,i-1$ it holds that
\begin{equation*}
    |\Ss(v,j)| = \left(1+o\left(\frac{1}{\log n}\right)\right)d^j.
\end{equation*}

Now, for each $v\in S$, we have $\Nn(v,i-1)=\{v\}\cup\bigcup_{j=1}^{i-1} \Ss(v,j)$, so 
if $S$ is good then
\begin{align*}
|\Nn(v,i-1)|&=1+\sum_{j=1}^{i-1} |\Ss(v,j)| \\
&=1+\sum_{j=1}^{i-1}\left(1+o\left(\frac{1}{\log n}\right)\right)d^j\\
&=\left(1+o\left(\frac{1}{\log n}\right)\right)d^{i-1},
\end{align*}
where the last equality follows from the fact that $d^j=o(d^{j+1}/\log n)$ for all $1 \le j < i-1$. 
Thus,
\begin{equation} \label{eqn:good_size}
|\Nn(S,i-1)|=\left|\bigcup_{v\in S} \Nn(v,i-1) \right|\leq \left(1+o\left(\frac{1}{\log n}\right)\right)d^{i-1}s,
\end{equation}
provided $S$ is good.

Our goal is to determine the size of the set $R=R(S):=V \setminus \Nn(S,i)$, the set of vertices that are at distance at least $i+1$ from every vertex of $S$. Note that at this point, edges within $\Nn(S,i-1)$ as well as edges between $\Nn(S,i-2)$ and $V \setminus \Nn(S,i-1)$ have been exposed. However, of greater importance is the fact that the edges between $\Nn(S,i-1) \setminus \Nn(S,i-2)$ and $V \setminus \Nn(S,i-1)$ have not as yet been exposed. As such, the edges of $G$ between $V \setminus \Nn(S,i-1)$ and $\Nn(S,i-1) \setminus \Nn(S,i-2)$ are independent from $\scr{E}$. Moreover, the set $R$ is exactly the set of vertices in $V \setminus \Nn(S,i-1)$ that are \emph{not} adjacent to any vertex in $\Nn(S,i-1) \setminus \Nn(S,i-2)$. Thus, $|R|$ conditional on $\scr{E}$ is distributed as a binomial
of parameters $|V \setminus \Nn(S,i-1)|$ and $(1-p)^{|\Nn(S,i-1) \setminus \Nn(S,i-2)|}$.
In particular, its expectation is
\[
\E \Big[ |R| \, \Big| \, \scr{E} \Big] =|V \setminus \Nn(S,i-1)| \cdot (1-p)^{|\Nn(S,i-1) \setminus \Nn(S,i-2)|}.
\]

Now, if $S$ is good, then we may apply \eqref{eqn:good_size} to ensure
that
\begin{align*}
|V \setminus \Nn(S,i-1)| &\cdot (1-p)^{|\Nn(S,i-1) \setminus \Nn(S,i-2)|} \\
&\geq (n-2d^{i-1}s)\cdot \left(1-\frac{d}{n}\right)^{\left(1+o\left(1/\log n\right)\right)d^{i-1}s} \\
&\sim n \cdot \exp\Big( -\left(1+o\left(1/\log n\right)\right)d^{i}s/n \Big) \\
&\sim n \cdot \exp \left(- \log d + 3\log\log n\right) = n \log^3 n / d = r.
\end{align*}
Thus,
\begin{align*} 
 \E \Big[ |R| \, | \, \scr{E} \Big] \cdot \bm{1}_{\text{$S$ is good}} &= |V \setminus \Nn(S,i-1)| \cdot (1-p)^{|\Nn(S,i-1) \setminus \Nn(S,i-2)|} \cdot \bm{1}_{\text{$S$ is good}} \\
            &\ge (1 + o(1)) \cdot r \cdot \bm{1}_{\text{$S$ is good}},
\end{align*}
so since $r = n \log^3 n / d$,
\begin{equation}\label{eqn:conditional_expectation_size}
\E \Big[ |R| \, | \, \scr{E} \Big] \cdot \bm{1}_{\text{$S$ is good}} \ge (1 +o(1)) \cdot n \log^3 n / d \cdot \bm{1}_{\text{$S$ is good}}.
\end{equation}

Let us now apply the Chernoff bound to $|R|$ conditional on $\scr{E}$.
Observe then that 
\begin{equation}\label{eqn:chernoff_size}
\Pr \left( |R| \leq r/2 \, \Big| \, \scr{E} \right)  \leq \exp \left( - \Theta \left( \E \Big[ |R| \, \Big| \, \scr{E} \Big] \right) \right).
\end{equation}
Thus, after multiplying each side of \eqref{eqn:chernoff_size} by $\bm{1}_{\text{$S$ is good}}$, it holds that
\begin{align*}
    \Pr (|R| \leq r/2 \, | \, \scr{E}) \cdot \bm{1}_{\text{$S$ is good}} &\leq  \exp(- \Theta(\E \Big[ |R| \, | \, \scr{E} \Big]) ) \cdot \bm{1}_{\text{$S$ is good}}\\
                                                                         &\leq \exp(- \Theta(n \log^3 n / d) ) \cdot \bm{1}_{\text{$S$ is good}},
\end{align*}
where the final line follows from \eqref{eqn:conditional_expectation_size}.
Moreover, whether or not $S$ is good can be determined by $\scr{E}$,
so 
\[
\Pr (|R| \leq r/2 \, | \, \scr{E}) \cdot \bm{1}_{\text{$S$ is good}} = \Pr (\text{$|R| \leq r/2$ and $S$ is good} \, | \, \scr{E}).
\]
Thus, after taking expectations we get that
\begin{equation}\label{eqn:good_set_upper_bound}
    \Pr (\text{$|R| \leq r/2$ and $S$ is good}) \leq \exp ( - \Theta(n \log^3 n / d)) \cdot \Pr ( \text{$S$ is good}).
\end{equation}

We must now show that $\aas$ for any set $S$, we have that $|R(S)| \geq r/2$.
Now, by applying Lemma \ref{lem:gnp exp} together with \eqref{eqn:good_set_upper_bound} and 
the union bound, we get that
\begin{align*}
    \Pr \left( \cup_{S \in \binom{V}{s}} |R(S)| \leq r/2 \right) &\le \Pr \left( \cup_{S \in \binom{V}{s}} \text{$|R(S)| \leq r/2$ and $S$ is good} \right) \\
    & \quad + \Pr \left( \cup_{S \in \binom{V}{s}} \text{$S$ is bad} \right)  \\
                                        &\le \binom{n}s \exp ( - \Theta(n \log^3 n / d )) + o(1).
\end{align*}
Observe however that
\begin{eqnarray*}
\binom{n}s \exp ( - \Theta(n \log^3 n / d )) &\leq& \left(\frac{ne}s\right)^s \exp ( - \Theta(n \log^3 n / d )) \\
&\leq& \exp(s\log n - \Theta(n \log^3 n /d ) )\\
&=& \exp(\Theta(n \log^2 n / d^i) - \Theta(n \log^3 n / d)) \\
&=& \exp(\Theta(- \Theta(n \log^3 n / d)) = o(1).
\end{eqnarray*}
It follows that a.a.s.\ $R(S)\geq r/2$ for all sets $S\in \binom{V}{s}$.
\end{proof}

Lemma \ref{lemma diametrically opposed} ensures that $\aas$, no matter where the cop decides
to place their sensors, the equivalence class with signature $(i+1, \ldots ,i+1)$ will be
at least size $r/2$. Thus, in order to survive the first round, the robber may choose
this equivalence class. The next lemma allows us to bound the number of vertices which are not reachable by the robber,
and thus ensure that the robber always has a feasible follow-up move.

\begin{lemma}\label{lemma r/4}
Suppose that $d:=p \cdot n$ is such that $\log n \ll d \ll n$. 
Let 
$$
r := \frac {n \log^3n}{d}.
$$
Then the following holds a.a.s.\ for $G = (V,E) \in \Gnp$: 
for every set $R\subseteq V$ with $|R|=r/4$, we have
\[
|V \setminus \Nn(R,1)|\leq r/4.
\]
\end{lemma}

\begin{proof}
Fix $R\in \binom{V}{r/4}$. Our goal is to estimate the size of the set $U=U(R):=V\setminus \Nn(R,1)$, that is, the set of vertices of $V \setminus R$ that are not adjacent to any vertex in $R$. Clearly,
\begin{align*}
    \Pr(|U|\geq r/4)&\leq \binom{|V\setminus R|}{r/4}\left((1-p)^{|R|}\right)^{r/4} \\
    &\leq \binom{n}{r/4}\left((1-p)^{r/4}\right)^{r/4}\\
    &\leq \left(\frac{4ne}r\right)^{r/4}\cdot\exp\left(-\frac{d}{n}\cdot\frac{r}4\cdot\frac{r}4\right)\\
    &\leq \exp\left(\frac{r}4\log n - \frac{r}{16}\log^3 n\right)\\
    &= \exp\left(- \Theta( r \log^3 n) \right).
\end{align*}
Hence, by the union bound, the probability that some set $U(R)$ does \emph{not} satisfy the desired bound for its size is at most
\[
\binom{n}{r/4}\exp\left(- \Theta( r \log^3 n) \right) \leq \exp\left( \Theta( r \log n) - \Theta( r \log^3 n) \right) = \exp\left(- \Theta( r \log^3 n) \right) = o(1).
\]
It follows that a.a.s.\ $|U(R)|\leq r/4$ for all sets $R\in\binom{V}{r/4}$.
\end{proof}

We are now ready to prove Theorem \ref{theorem lower bound}. 

\begin{proof}[Proof of Theorem \ref{theorem lower bound}]
Since we aim for a statement that holds \aas\ we may assume that we have a deterministic graph $G$ that satisfies the properties in the conclusions of Lemmas~\ref{lemma diametrically opposed},~\ref{lemma r/4} and~\ref{lem:diameter}. The strategy for the robber is simple; he always stays in the equivalence class of vertices whose $S_j$-signature is $(i+1,i+1,\dots,i+1)$. 

Let $r:=n \log^3 n / d$. Assume the cops first choose a set $S_1$ of size $s$ as the sensor locations. Combining Lemma~\ref{lemma diametrically opposed} and Lemma~\ref{lem:diameter} we see that the equivalence class of vertices with $S_1$-signature $(i+1,i+1,\dots,i+1)$, call it $X_1=V_1$ is of size at least $r/2 \ge r/4$. Indeed, Lemma~\ref{lemma diametrically opposed} provides an upper bound for the size of \emph{all} equivalence classes with at least one value at most $i$ in their signatures. Lemma~\ref{lem:diameter} guarantees that the only other equivalence class is the one with signature $(i+1,i+1,\dots,i+1)$. The robber will choose this equivalence class. 
We can continue iteratively: for $j\in \N$, assume that the robber has chosen a set $V_j$ of size at least $r/4$, and the cops respond with sensors at the set $S_{j+1}$. Then let $X_{j+1}$ be the set of all vertices with $S_{j+1}$-signature $(i+1,i+1,\dots,i+1)$. By Lemma~\ref{lemma diametrically opposed} and Lemma~\ref{lem:diameter}, $|X_{j+1}|\geq r/2$, and by Lemma~\ref{lemma r/4}, $V_{j+1}:=N(V_j,1) \cap X_{j+1}$ is of size at least $r/4$. Thus the robber can always choose the equivalence class of vertices with signature $(i+1,i+1,\dots,i+1)$, and this class will always be of size at least $r/4$. This shows that the cops will never be able to locate the robber.
\end{proof}

\section{Upper Bound}\label{sec:upper}

In this section, we will prove two upper bounds. The first one will apply to random graphs with $p \cdot n \gg \log n$, the diameter equal to $i+1$, and when $d^i = o(n)$. The second bound will cover the remaining cases, provided that $p \cdot n \gg \log^3 n$. In the previous section, the robber was able to employ a greedy strategy of always staying diametrically opposite to all the sensors. In order to prove the upper bound, we will not give an explicit strategy for the cops, but instead we will use the probabilistic method to show that there exists a winning strategy for the cops.

\begin{theorem}\label{theorem upper bound}
Suppose that $d:=p \cdot n$ is such that $\log n \ll d \ll n$. Let $i=i(n) \in \N$ be the largest integer $i$ such that $d^i\ll n$. If $d^{i+1}/n-2\log n \to \infty$, then the following holds a.a.s.\ for $G \in \Gnp$: 
\[
\zeta(G)\le (1+o(1)) \left(\log d + 2 \log\log n\right)\frac{n}{d^i}\,.
\]
\end{theorem}

\begin{proof} 
In fact, we will prove something slightly stronger. Let 
$$
\omega=\omega(n) := \min \left\{ \frac {d}{\log n}, \frac {n}{d^i}, (\log n)^4 (\log \log n)^2 \right\}\,.
$$
Since $d \gg \log n$ and $d^i \ll n$, we get that $\omega \to \infty$ as $n \to \infty$.
Suppose that $G_n = (V_n, E_n)$ is a family of graphs which satisfies the following properties: 

For each $n \in \N$
\begin{itemize}
    \item [(a)] $|V_n| = n$,
    \item [(b)] for every $x,y \in V_n$ ($x\neq y$) and $j \in \N$ such that $1 \le j \le i$ we have that
    $$
    |\Ss(x,j) \setminus \Ss(y,j)| = (1 + O(1/\sqrt{\omega})) \cdot d^j,
    $$
    \item [(c)] the diameter of $G_n$ is $i+1$,
    \item [(d)] the maximum degree of $G_n$ is $(1+o(1)) \cdot d$.
\end{itemize}
Assuming these conditions, there exists some $N\in\mathbb{N}$ (which depends only on the bounds in (b) and (d), and not on the family $G_n$) such that for all $n\geq N$, (deterministically)
$$
\zeta(G_n) \le k:= \left(1+\frac {1}{\omega^{1/3}} \right)(\log d + 2\log\log n) \frac {n}{d^i} \sim (\log d + 2\log\log n) \frac {n}{d^i}.
$$

The result will follow from Lemma~\ref{lem:gnp exp} (that shows that $\Gnp$ satisfies property (b) and (d) a.a.s.\ with a uniform choice of error function) and Lemma~\ref{lem:diameter} (that shows that property (c) is satisfied a.a.s.). Indeed, Lemma~\ref{lem:gnp exp} can be applied as $d \ge \omega \log n$, $d^i/n \le 1/\omega = O(1/\sqrt{\omega})$, and $\omega \le (\log n)^4 (\log \log n)^2$--see the definition of $\omega$. Lemma~\ref{lem:diameter} can be applied as $d^i/n-2\log n = o(1)-2\log n \to -\infty$ and, by assumption, $d^{i+1}/n-2\log n \to \infty$.

\smallskip

Let us then concentrate on a deterministic family of graphs $G_n=(V_n,E_n)$ satisfying propeties (a)-(d). Recall that in Section~\ref{sec:perfect} we reformulated the game so that it can be viewed as a perfect information game, and so we may assume that the moves of the robber are guided by a perfect player that has a fixed strategy for a given deterministic graph $G_n$. In particular, the robber chooses sets $(R^{1}_{j_1},R^{2}_{j_2},\dots,R^{i}_{j_i})$ in response to the cop choosing sensor locations $(S_1,S_2,\dots,S_i)$. Such responses are predetermined before the game actually starts. See Section~\ref{sec:perfect} for more details. 

On the other hand, to get an upper bound for the localization number, the cops are going to use a simple strategy, namely, at each round $t$ of the game, the cops choose a random set $S_t \subseteq V_n$ of cardinality $k$ for the sensor locations (regardless of anything that happened during the game thus far). Clearly, this is a sub-optimal strategy but, perhaps surprisingly, it turns out that it works very well.

Trivially, $|\Nn(R^{1}_{j_1},1)| \le n$. Our goal is to show that with high probability, for each round $t$, we have
$$
|\Nn(R^{t+1}_{j_{t+1}},1)| \le |\Nn(R^{t}_{j_t},1)| / \log n.
$$ 
As a result, this bound will hold a.a.s.\ for $1 \le t \le t_F := \log n / \log \log n$, and so $|\Nn(R^{t+1}_{j_{t+1}},1)| \le n/\log^t n$. In particular, we will get that $|\Nn(R^{t_F+1}_{j_{t_F+1}},1)| \le 1$ and so the cops win before the end of round $t_F+1$.

Suppose that at some round $t$, the robber ``occupies'' the set $R^{t}_{j_t}$ in response to the cop choosing sensor locations $(S_1,S_2,\dots,S_t)$. As mentioned above, the cops choose the set $S_{t+1}$ at random. It would be convenient to generate this random set as follows: select $k$ vertices to form $S_{t+1}$ one by one, each time choosing a random vertex with uniform probability from the set of vertices not selected yet. Once $S_{t+1}$ is fixed, the set $\Nn(R^{t}_{j_t},1)$ is partitioned into sets having the same $S_{t+1}$-signatures. The robber then has to pick $R^{t+1}_{j_{t+1}}$, one of the equivalence classes. We will show that, regardless of her choice, $|\Nn(R^{t+1}_{j_{t+1}},1)| \le |\Nn(R^{t}_{j_t},1)| / \log n$ will hold with high probability.  

There are 
$$
\binom{|\Nn(R^{t}_{j_t},1)|}{2} \le  |\Nn(R^{t}_{j_t},1)|^2
$$ 
pairs of vertices. Let us focus on one such pair, $x,y$, and suppose that the cops put a sensor on some vertex $v \in V_n$. Note that this pair is distinguished by $v$ \emph{if and only if} $v$ belongs to the set
\begin{eqnarray*}
D(x,y) &:=& \bigcup_{j \ge 0} \Big( \Ss(x,j) \setminus \Ss(y,j) \Big) \cup \Big( \Ss(y,j) \setminus \Ss(x,j) \Big) \\
&=& \bigcup_{j = 0}^{i} \Big( \Ss(x,j) \setminus \Ss(y,j) \Big) \cup \Big( \Ss(y,j) \setminus \Ss(x,j) \Big).
\end{eqnarray*}
Indeed, if $v \in \Ss(x,j) \setminus \Ss(y,j)$, then the distance between $v$ and $x$ is $j$ but the distance between $v$ and $y$ is not. Moreover, since the diameter of $G_n$ is $i+1$ (property~(c)), in order to distinguish the pair $x,y$, the distance from $v$ to at least one of $x,y$ has to be at most $i$. This justifies the equality above. By Property~(b), we may estimate the size of the distinguishing set as follows:
\[
|D(x,y)| ~=~ \sum_{j=0}^i (1+O(1/\sqrt{\omega})) 2 d^j ~=~ (1+O(1/\sqrt{\omega})) 2 d^i.
\]
The probability that the pair cannot be distinguished by any of the sensors in $S_{t+1}$ is at most
\begin{eqnarray*}
\Big( 1 - |D(x,y)|/n \Big)^k &=& \Big( 1 - (1+O(1/\sqrt{\omega})) 2 d^i/n \Big)^k \\
&=& \exp \Big( - (1+O(1/\sqrt{\omega})) 2 d^i k /n \Big) \\
&=& \exp \Big( - (1+1/\omega^{1/3})(1+O(1/\sqrt{\omega})) 2( \log d + 2 \log \log n) \Big) \\
&\le& \exp \Big( - 2( \log d + 2 \log \log n) \Big) = \frac {1}{d^2 \log^4 n}\,.
\end{eqnarray*}

Let $X_{t+1}$ be the number of pairs in $\Nn(R^{t}_{j_t},1)$ with the same signature in $S_{t+1}$. Since $\E[X_{t+1}] \le |\Nn(R^{t}_{j_t},1)|^2 d^{-2} \log^{-4} n$, it follows immediately from Markov's inequality that $X_{t+1} \le |\Nn(R^{t}_{j_t},1)|^2 d^{-2} \log^{-3} n$ with probability at least $1-1/\log n$. If this bound holds, then we say the round is \emph{good}. If this is the case, then, regardless which equivalence class of the partition of $\Nn(R^{t}_{j_t},1) = R^{t+1}_{1} \cup R^{t+1}_{2} \cup \ldots \cup R^{t+1}_{\ell_{t+1}}$ the robber selects as her response, the selected set $R^{t+1}_{j_{t+1}}$ is of size at most $2 \sqrt{X_{t+1}} \le 2 |\Nn(R^{t}_{j_t},1)| d^{-1} \log^{-3/2} n$. Indeed, note that
$$
X_{t+1} = \sum_{j=1}^{\ell_{t+1}} \binom{|R^{t+1}_{j}|}{2} \ge \binom{|R^{t+1}_{j_{t+1}}|}{2} \ge |R^{t+1}_{j_{t+1}}|^2/4.
$$
Finally, since the maximum degree of $G_n$ is asymptotic to $d$ (property (d)), the closed neighbourhood of $R^{t+1}_{j_{t+1}}$ has the size at most 
$$
(2 + o(1)) |\Nn(R^{t}_{j_t},1)| \log^{-3/2} n \le |\Nn(R^{t}_{j_t},1)| \log^{-1} n, 
$$
as required.

It remains to show that a.a.s.\ the first $t_F = \log n / \log \log n$ rounds are good. Since each round is \emph{not} good with probability at most $1/\log n$, the probability that some round is not good is at most $t_F/\log n = o(1)$, and the proof is finished. We get that this randomized strategy for $k$ cops works a.a.s.\ and so the probability it works is larger than, say, 1/2 for $n$ sufficiently large. It follows that the cops have a winning strategy and so the claimed bound for $\zeta(G_n)$ holds deterministically.  
\end{proof}

Before we move to the upper bound that covers the remaining cases, let us briefly discuss why the bound changes. The size of the set $D(x,y)$ defined in the proof above that distinguishes the pair of vertices $(x,y)$ plays an important role in the proof---the larger the set, the smaller the upper bound we get. We noticed that 
$$
s = s(n) := |D(x,y)| = \sum_{j \ge 0} s_j, 
$$
where
$$
s_j := \left| \Big( \Ss(x,j) \setminus \Ss(y,j) \Big) \cup \Big( \Ss(y,j) \setminus \Ss(x,j) \Big) \right|.
$$

Suppose that $d \gg \log^3 n$. Let $i=i(n) \in \N$ be the largest integer $i$ such that $d^i \le 3 \log n$, and let $c = c(n) = d^i / n$. The previous bound, Theorem~\ref{theorem upper bound}, applies to the case when $c = o(1)$; in particular, the diameter is equal to $i+1$ a.a.s. For this case, $s_i$ is the dominating term in the sum: $s \sim s_i \sim 2d^{i}$. If $c \to A \in (0,\infty)$, then $s \sim s_i \sim 2n(1-e^{-A})e^{-A}$; in particular, $s$ increases when $A \in (0,\log 2)$ reaching its maximum at $(1/2+o(1))n$ but then it starts decreasing when $A \in (\log 2, \infty)$. When $c \to \infty$ but 
$$
c - ( \log d - \log \log d) \to B \in \R,
$$
then $s$ is dominated by two terms: $s_{i-1} \sim 2d^{i-1}$, and
$$
s_i \sim 2 n e^{-c} \sim \frac{2n (\log d) e^{-B}}{d}  \sim  \frac {2nce^{-B}}{d} = 2d^{i-1} e^{-B}.
$$
It follows that $s \sim s_{i-1}+s_i \sim 2d^{i-1}(1+e^{-B}) \sim 2ne^{-c}(e^B+1)$. In particular, $s \sim 2ne^{-c}$ when $B \to -\infty$ and $s \sim 2d^{i-1}$ when $B \to \infty$. Here is the summary of our observations:
$$
s \sim
\begin{cases}
2 d^i & \text{ if } c = o(1) \\
2n(1-e^{-A})e^{-A} & \text{ if } c \to A \in \R_+ \\
2ne^{-c} & \text{ if } c \to \infty \text{ and } c - (\log d - \log \log d) \to -\infty \\
2d^{i-1}(1+e^{-B}) = 2ne^{-c} (e^B+1) & \text{ if } c - (\log d - \log \log d) \to B \in \R \\
2d^{i-1} & \text{ if } c - (\log d - \log \log d) \to \infty \text{ and } c \le 3 \log n.
\end{cases}
$$

\bigskip

We are now ready to cover the remaining cases that Theorem~\ref{theorem upper bound} did not cover, and finalize the upper bound.

\begin{theorem}\label{theorem upper bound2}
Suppose that $d:=p \cdot n$ is such that $\log^3 n \ll d \ll n$. Let $i=i(n) \in \N$ be the largest integer $i$ such that $d^i\le 3 \log n$, and $c = c(n) = d^i/n$. Then, the following holds a.a.s.\ for $G \in \Gnp$.
\begin{itemize}
    \item [(i)] if $c \to A \in \R_+$, then
\[
\zeta(G)\le (1+o(1)) \left(\log d + 2 \log\log n\right)\frac {e^A}{1-e^{-A}}\,.
\]
    \item [(ii)] if $c \to \infty$ and $c - (\log d - \log \log d) \to -\infty$, then
\[
\zeta(G)\le (1+o(1)) \left(\log d + 2 \log\log n\right) e^c\,.
\]
    \item [(iii)] if $c - (\log d - \log \log d) \to B \in \R$, then
\begin{eqnarray*}
\zeta(G) &\le& (1+o(1)) \left(\log d + 2 \log\log n\right) \frac {e^c}{e^B+1} \\
&\sim&  \left(\log d + 2 \log\log n\right) \frac {n}{d^{i-1}(1+e^{-B})}\,.
\end{eqnarray*}
    \item [(iv)] if $c - (\log d - \log \log d) \to \infty$ and $c \le 3 \log n$, then
\[
\zeta(G)\le (1+o(1)) \left(\log d + 2 \log\log n\right) \frac {n}{d^{i-1}}\,.
\]
\end{itemize}
\end{theorem}

\begin{proof}
The proof of this theorem is almost identical to the one of Theorem~\ref{theorem upper bound}, and so we will only highlight differences. We will use the definitions of $s$ and $s_j$ that we introduced right before the statement of this theorem. We used Lemma~\ref{lem:gnp exp} to estimate $s$ in Theorem~\ref{theorem upper bound} but this time we will also need Lemma~\ref{lem:gnp exp2}. As the asymptotic behaviour of $s$ changes, we will need to adjust $k$ accordingly. However, in each case, $k \sim 2n(\log d + 2 \log \log n)/s$. We will deal with each case independently.

\medskip

For part (i), after setting 
$$
\omega' = \omega'(n) = \min \left\{ \frac {d}{\log^3 n}, \frac {\log^2 n}{(\log \log n)^2} \right\},
$$
we get that 
\begin{eqnarray*}
s &=& \sum_{j=0}^i s_j = \sum_{j=0}^{i-1} s_j + s_i = (1+o(1)) 2 d^{i-1} + (1+O(1/\sqrt{\omega'})) 2n (1-e^{-A})e^{-A} \\ 
&=& (1+O(1/\sqrt{\omega'})) 2n (1-e^{-A})e^{-A},
\end{eqnarray*}
and so the upper bound has to be adjusted to 
$$
k := \left(1+\frac {1}{\omega'^{1/3}} \right) (\log d + 2 \log \log n) \frac {e^A}{1-e^{-A}}.
$$

For part (iii), we set
$$
\omega' = \omega'(n) = \min \left\{ \frac {d}{\log^3 n}, \log \log n \right\},
$$
and observe that 
\begin{eqnarray*}
s_i &=& (1+O(1/\sqrt{\omega'})) 2n (1-e^{-c})e^{-c} \\ 
&=& (1+O(1/\sqrt{\omega'})) 2n (1+O(1/c))e^{-c} \\
&=& (1+O(1/\sqrt{\omega'})) 2n e^{-c} \\
&=& (1+O(1/\sqrt{\omega'})) 2n d^{-1} (\log d) e^{-B} \\
&=& (1+O(1/\sqrt{\omega'})) 2n d^{-1} c (1+O(\log \log d/\log d)) e^{-B} \\
&=& (1+O(1/\sqrt{\omega'})) 2 cn d^{-1} e^{-B} \\
&=& (1+O(1/\sqrt{\omega'})) 2 d^{i-1} e^{-B}.
\end{eqnarray*}
On the other hand, 
$$
\sum_{j=0}^{i-1} s_j = (1+O(1/d)) s_{i-1} = (1+O(1/(\sqrt{\omega'}\log n))) 2d^{i-1} = (1+O(1/\sqrt{\omega'})) 2d^{i-1}.
$$
It follows that $s = (1+O(1/\sqrt{\omega'})) 2d^{i-1} (1+e^{-B})$, and so the upper bound has to be adjusted to 
$$
k := \left(1+\frac {1}{\omega'^{1/3}} \right) (\log d + 2 \log \log n) \frac {n}{d^{i-1} (1-e^{-B})}.
$$

For part (ii), we observe that $s_i / s_{i-1} = e^{-B} \to \infty$. One can apply a trivial bound $s \ge s_i = (1+O(1/\sqrt{\omega'})) 2 d^{i-1} e^{-B}$ and adjust the definition of $k$ to get the desired bound. Similarly, for part (iv), we observe that $s_i / s_{i-1} = e^{-B} \to 0$. (In fact, if $c - 2 \log n \to \infty$, then a.a.s.\ the diameter of a random graph is $i$ and so there is no need to consider $s_i$ anymore.) This time we use the fact that $s \ge s_{i-1} = (1+O(1/\sqrt{\omega'})) 2d^{i-1}$. The claimed bound holds and the proof is finished.
\end{proof}


\begin{thebibliography}{99}

\bibitem{bol} B.\ Bollob\'{a}s, \emph{Random Graphs}, Cambridge University Press, Cambridge, 2001.

\bibitem{metric_dimension} B.\ Bollob\'{a}s, D.\ Mitsche, and P.\ Pra\l{}at, Metric dimension for random graphs, Electronic Journal of Combinatorics 20(4) (2013), \#P1.

\bibitem{Anthony_Bill} A.~Bonato and W.B.~Kinnersley, Bounds on the localization number, preprint.

\bibitem{book} A.~Bonato and P.~Pra\l{}at, Graph Searching Games and Probabilistic Methods, CRC Press, 2017.

\bibitem{paper5} B.~Bosek, P.~Gordinowicz, J.~Grytczuk, N.~Nisse, J.~Sok\'o\l{}, M.~\'Sleszy\'nska-Nowak, Localization game on geometric and planar graphs, Discrete Applied Mathematics 251 (2018) 30--39.

\bibitem{paper7} A.~Brandt, J.~Diemunsch, C.~Erbes, J.~LeGrand, C.~Moffatt, A robber locating strategy for trees, Discrete Applied Mathematics 232 (2017) 99--106.

\bibitem{paper9} J.~Carraher, I.~Choi, M.~Delcourt, L.H.~Erickson, D.B.~West, Locating a robber on a graph via distance queries, Theoretical Computer Science 463 (2012) 54--61.

\bibitem{Dudek} A.~Dudek, A.~Frieze, and W.~Pegden, A note on the localization number of random graphs: diameter two case, Discrete Applied Mathematics 254, 107--112.

\bibitem{paper16} J.~Haslegrave, R.A.B.~Johnson, S.~Koch, Locating a robber with multiple probes, Discrete Mathematics 341 (2018) 184--193.

\bibitem{JLR}  S.~Janson, T.~{\L}uczak, and A.~Ruci\'nski, {\em Random graphs}, Wiley, New York, 2000.     

\bibitem{Odor2019} G. \'{O}dor and P. Thiran, Sequential Metric Dimension for Random Graphs, \textit{arXiv prepint arXiv:1910.10116}.

\bibitem{paper19} S.~Seager, Locating a robber on a graph, Discrete Mathematics 312 (2012) 3265--3269.

\bibitem{paper20}  S.~Seager, Locating a backtracking robber on a tree, Theoretical Computer Science 539 (2014) 28--37.




\end{thebibliography}
\end{document}